\documentclass[a4paper,12pt]{amsart}
\usepackage{amssymb}
\usepackage{xypic}
\usepackage[all,cmtip]{xy}
\usepackage{amsmath}
\usepackage{chemarrow}
\usepackage{hyperref}
\usepackage{mathrsfs}
\usepackage{wasysym}
\usepackage{titletoc}
\usepackage{bm}
\usepackage{geometry}
\usepackage{stmaryrd}
\usepackage{colordvi}
\usepackage{color}

\allowdisplaybreaks

\DeclareMathOperator{\End}{End}
\DeclareMathOperator{\Ext}{Ext}
\DeclareMathOperator{\gld}{gld}

\DeclareMathOperator{\hdet}{hdet}

\DeclareMathOperator{\Hom}{Hom}
\DeclareMathOperator{\HOM}{HOM}

\DeclareMathOperator{\pd}{pd}

\DeclareMathOperator{\Tor}{Tor}

\DeclareMathOperator{\Z}{Z}

\DeclareMathOperator{\id}{id}

\newcommand{\To}{\longrightarrow}

\numberwithin{equation}{section}

\theoremstyle{definition}

\newtheorem{thm}{Theorem}[section]
\newtheorem{prop}[thm]{Proposition}
\newtheorem{lem}[thm]{Lemma}

\newtheorem{cor}[thm]{Corollary}
\newtheorem{defn}[thm]{Definition}
\newtheorem{rk}[thm]{Remark}

\newtheorem{ex}[thm]{Example}

\begin{document}
	
\title{Yoneda Ext-algebras of Takeuchi smash products}
	


\author{Quanshui Wu}
\address{School of Mathematical Sciences, Fudan University, Shanghai 200433, China}
\email{qswu@fudan.edu.cn}

\author{Ruipeng Zhu}
\address{Department of Mathematics, Southern University of Science and Technology, Shenzhen, Guangdong 518055, China}
\email{zhurp@sustech.edu.cn}

\begin{abstract}
We prove that the Yoneda Ext-algebra of a Takeuchi smash product is the graded Takeuchi smash product of the Yoneda Ext-algebras of the two algebras or modules involved. As an application, we prove that graded Takeuchi smash products preserve Artin-Schelter regularity, and describe the Nakayama automorphism of the product.
\end{abstract}


\thanks{}

\maketitle

\section*{Introduction}

Let $H$ be a Hopf algebra with bijective antipode, $A$ be a left $H$-module algebra and $B$ be a left $H$-comodule algebra. There is a canonical algebra construction on the space $A\otimes B$ called the Takeuchi smash product of $A$ and $B$, denoted by $A\#B$ (see Definition \ref{defi-Take-product}).
Usual smash products and Ore extensions are examples of Takeuchi smash products (see Example \ref{ore-ext-Hopf-alg}).
%

If $A$ is a positively graded algebra and $G$ is a finite group acting gradedly on $A$,  
Martinez-Villa \cite[Theorem 10]{Mar} proved that the Yoneda Ext-algebra of the skew  group algebra $A\#G$ is isomorphic to the skew group algebra $\Ext^\bullet_A(A_0, A_0) \# G$ under some mild conditions. 
Bergh and Oppermann proved that the Yoneda Ext-algebra of a class of Takeuchi smash product (see Example \ref{a-class-twisted-tensor-product}) is isomorphic to the Takeuchi smash product of the Yoneda Ext-algebras \cite[Theorem 3.7]{BO}.
The main purpose of this paper is to give a precise description of the Yoneda Ext-algebras of the general Takeuchi smash products.

Let $M$ be a left $A\#H$-module and $X$ be a left $(H, B)$-Hopf module. Then there is a natural left $A\#B$-module structure on $M\otimes X$ denoted by $M\#X$, which is called  the Takeuchi smash product module of $M$ and $X$. If $P_\bullet \to M \to 0$ is a projective $A\#H$-module resolution of $M$, then $\Hom^\bullet_A(P_\bullet, P_\bullet)$ is a differential graded left $H$-module algebra and 
$\Ext^\bullet_{A}(M, M) = \bigoplus_{n \geq 0} \Ext^n_{A}(M, M)$ is a graded left $H$-module algebra. If $X$ has a resolution $Q_\bullet \to X \to 0$ in the category of left $(H, B)$-Hopf modules such that each $Q_i$ is finitely generated as $B$-module, then $\Hom^\bullet_A(Q_\bullet, Q_\bullet)$ is a differential graded left $H$-comodule algebra and $\Ext^\bullet_{B}(X, X)$ is a graded left $H$-comodule algebra. Therefore, we can form the graded Takeuchi smash product algebra $\Ext^\bullet_{A}(M, M) \# \Ext^\bullet_{B}(X, X)$.

The following is one of the main results in this paper, which generalizes the result in the usual tensor product case,  and \cite[Theorem 10]{Mar}, \cite[Theorem 3.7]{BO}. Theorem \ref{Ext-alg of smash product} is the  right-sided version of it. 

\begin{thm} \label{thm1}
	Suppose that $M$  is a left $A\#H$-module which is pseudo-coherent as $A$-module, and $X$ is a left $(H, B)$-Hopf module which is pseudo-coherent as $B$-module. Then, as graded algebras,
	$$\Ext^\bullet_{A\#B}(M\#X, M\#X) \cong \Ext^\bullet_{A}(M, M) \# \Ext^\bullet_{B}(X, X)$$
	where the last $\#$ means graded Takeuchi smash product.
\end{thm}
A module is called pseudo-coherent if it has a finitely generated projective resolution.

The other result in this paper is the graded Takeuchi smash product of AS-regular algebras is AS-regular (see Theorem \ref{AS-regular-TSP}). In the case that $H$ is trivial or the actions of $H$ is trivial, the result is well known.

\begin{thm} \label{thm2}  Let $A$ and $B$ be AS-regular algebras of dimension $d_1$ and $d_2$ respectively. If $H$ is a positively graded Hopf algebra, $A$ is an $H$-module graded algebra and $B$ is an $H$-comodule graded algebra, 
then $A\#B$ is an AS-regular algebra of dimension $d_1+d_2$.
\end{thm}
The Nakayama automorphism of $A\#B$ is described as an application of Theorems \ref{thm1} and \ref{thm2}.

The paper is organized as follows.
In section $1$, we define Takeuchi smash products of algebras or modules, and give some basic facts.
In section $2$, we study homological properties of Takeuchi smash products, in particular, the Tor-groups and Ext-groups.
In section $3$, we introduce Takeuchi smash products of differential graded algebras, and prove Theorem \ref{Ext-alg of smash product}, that is, the Yoneda Ext-algebra of the Takeuchi smash product of modules is the graded Takeuchi smash product of the Yoneda Ext-algebras of the modules involved.
In section $4$, we prove that Takeuchi smash product preserves the AS-regularity, and describe the Nakayama automorphism of the product.

\section{Takeuchi smash products}

Throughout, let $k$ be a field. All algebras and modules are over $k$, and the unmarked tensor product $\otimes$ (resp. $\Hom$) means $\otimes_k$ (resp. $\Hom_k$).
For the basic concept and theory concerning Hopf algebras we refer to \cite{Sw} and \cite{Mo} as basic references.

\subsection{Takeuchi smash products of algebras}
Let $H$ be a Hopf algebra, $A$ be a left $H$-module algebra and $B$ be a left $H$-comodule algebra \cite[Definitions 4.1.1, 4.1.2]{Mo}. The following construction (see \cite[section 8]{Tak} and the references therein) is a generalization of smash products, which is now called Takeuchi smash product in literature.

\begin{defn} \label{defi-Take-product} The Takeuchi smash product $A\#B$ of $A$ and $B$ is the associative algebra with  underlying $k$-space $A\otimes B$, and multiplication defined by
$$(a\#b) (a'\#b') := \sum_{(b)}a(b_{-1} \rightharpoonup a') \# b_0b'$$
for any $a,a' \in A$ and $b,b' \in B$, where $\sum_{(b)} b_{-1} \otimes b_0 = \rho(b)$ is given by the comodule structure map of $B$.
\end{defn}


If $B=H$, the Takeuchi smash product $A\#B$ is the usual smash product $A\#H$ \cite[Definition 4.1.3]{Mo}.

\begin{ex}\label{ore-ext-Hopf-alg}
Any Ore extension $A[x, \sigma, \delta]$ can be viewed as a Takeuchi smash product, where $A$ is a $k$-algebra, $\sigma$ is an algebra automorphism of $A$ and $\delta$ is a $\sigma$-derivation of $A$.

In fact, let $H = k\langle g^{\pm 1}, X\rangle$ be the free algebra with the Hopf algebra structure given by
$$\Delta(g) = g\otimes g, \Delta(X) = g \otimes X + X \otimes 1, \varepsilon(g) = 1, \varepsilon(X) = 0, S(g) = g^{-1}, S(X) = -g^{-1}X.$$
%
%
%
Then  $A$ has a left $H$-module structure, given by $g \rightharpoonup a = \sigma (a) \, \textrm{ and } \,   X  \rightharpoonup a =  \delta(a).$
It is easy to see that
	$$g \rightharpoonup (ab) = (g \rightharpoonup a)(g \rightharpoonup b) \text{ and } X \rightharpoonup (ab) = (g \rightharpoonup a)(X \rightharpoonup b) + (X \rightharpoonup a)b$$
for all $a,b \in A$. So $m_A$ and $u_A$ are $H$-module morphisms, that is, with the action defined above $A$ is a left $H$-module algebra.
	
Let $B=k[X]$ be the subalgebra of $H$ generated by $X$. In fact, $B$ is a left coideal subalgebra of $H$. The Ore extension $A[x;\sigma, \delta]$ is isomorphic to the Takeuchi smash product $A\#B$, which is a subalgebra of $A \# H$.
\end{ex}

\begin{ex}\label{a-class-twisted-tensor-product}
	A tensor product of graded algebras twisted by a bicharacter can also be viewed as a Takeuchi smash product.
	Let $A$ and $B$ be two $k$-algebras that are graded by abelian groups $G$ and $L$ respectively:
	$A = \bigoplus_{g \in G} A_g$ and $B = \bigoplus_{l \in L} B_l.$

	Let $t: G \times L \to k^{\times}$ be a bicharacter, that is,
	$$t(g + g', l) = t(g, l)t(g', l)\, \text{ and } \, t(g, l + l') = t(g, l)t(g, l')$$
	for any $g, g' \in G$ and $l, l' \in L$.
	
	With this data, the twisted tensor product $A \otimes^t B$ (see \cite[Definition/Construction 2.2]{BO}) is by definition the vector space $A \otimes B$, with multiplication given by
	$$(a \otimes b) (a' \otimes b') = t(g, l)\, aa' \otimes bb',$$
	where $a \in A$, $a' \in A_g$, $b \in B_l$ and $b' \in B$.
	
	In fact, $B$ is a $kL$-comodule algebra, and $A$ is a $kL$-module algebra via $l \rightharpoonup a := t(g, l)a$ for any $g \in G$, $l \in L$ and $a \in A_g$.
	The twisted tensor product $A \otimes^t B$ is the Takeuchi smash product $A\#B$ over the Hopf algebra $kL$.
\end{ex}

In general, any Takeuchi smash product is a twisted tensor product \cite{CSV}.
Let $A$ and $B$ be two algebras, and $\tau: B \otimes A \longrightarrow A \otimes B$ be a linear map.
 Then there is a multiplication map on the vector space $A\otimes B$ given by
$$ m_{A\otimes^{\tau}B} := (m_A \otimes m_B) \circ (\id_A \otimes \tau \otimes \id_B)$$
where $m_A$ and $m_B$ are the multiplication maps of $A$ and $B$ respectively. If, with this multiplication, $A\otimes B$ is an associative algebra with unit $1_A \otimes 1_B$, then it is called a  {\it twisted tensor product} determined by $\tau$ and is denoted by  $A\otimes^{\tau}B$.
If $A\#B$ is a Takeuchi smash product, then it is the twisted tensor product $A\otimes^{\tau}B$ given by $\tau: B \otimes A \to A \otimes B, \, b \otimes a \mapsto \sum_{(b)} (b_{-1} \rightharpoonup a) \otimes b$.

\subsection{Takeuchi smash product of modules} Let $_A\mathcal{M}$ (resp. $\mathcal{M}_A$) denote the left (resp. right) $A$-module category for an algebra $A$.
Let $A$ be a left $H$-module algebra and $B$ be a left $H$-comodule algebra. Recall that a left $B$-module $X$ with a left $H$-coaction $\rho$ is called a left $(H,B)$-Hopf module \cite[Definition 8.5.1]{Mo} if $\rho(bx)=\rho(b)\rho(x)$ for all $b \in B$ and $x \in X$. Let $^H_B\mathcal{M}$ denote the left $(H,B)$-Hopf module category. The $(H,B)$-Hopf module category $^H\mathcal{M}_B$ is defined analogously.

Note that $A\#B \to (A\#H) \otimes B, a\#b \mapsto \sum_{(b)}(a\#b_{-1}) \otimes b_0$ is an algebra map.

For any $M \in {_{A\#H}\mathcal{M}}$ and $X \in {_{B}\mathcal{M}}$, there is a natural left $A\#B$-module structure on $M \otimes X$, given by
\begin{equation} \label{left AxB-mod-on-tensor}
a\#b \rightharpoonup m\otimes x = \sum_{(b)}a(b_{-1} \rightharpoonup m) \otimes b_0x = \sum_{(b)} (a \# b_{-1})m \otimes b_0x
\end{equation}
for any $a \in A$, $b \in B$, $m \in M$ and $x \in X$. This module, denoted  by $M\#X$, is called the {\it Takeuchi smash product} of $_{A\#H}M$ and $_BX$, where the element $m \otimes x$ in $M\#X$ is denoted by $m\#x$.

For any $M \in \mathcal{M}_{A\#H}$ and $X \in \mathcal{M}_B$, there is also a right $A\#B$-module structure on $M \otimes X$  given by
\begin{equation} \label{right AxB-mod-on-tensor-2}
(m \otimes x) (a \# b) = \sum_{(b)} m(a\#b_{-1}) \otimes xb_0.
\end{equation}	
This module, denoted also by $M\#X$, is called the {\it Takeuchi smash product} of  $M_{A\#H}$ and $X_B$.

For any $N \in \mathcal{M}_A$ and $Y \in {^{H}\mathcal{M}_B}$, there is a natural right $A\#B$-module structure on $N \otimes Y$, given by
\begin{equation}\label{right AxB-mod-on-tensor}
n\otimes y \leftharpoonup a\#b = \sum_{(y)} n(y_{-1}\rightharpoonup a) \otimes y_0b
\end{equation}
for any $a \in A$, $b \in B$, $n \in N$ and $y \in Y$. This module, denoted  by $N\#Y$, is called the {\it Takeuchi smash product} of $N_A$ and $Y \in {^{H}\mathcal{M}_B}$, where the element $n \otimes y$ in $N\#Y$ is denoted by $n\#y$.

For any $N \in {_{A}\mathcal{M}}$ and $Y \in {_B^H\mathcal{M}}$, there is a left $A\#B$-module structure on $N \otimes Y$ which is given by
\begin{equation}
\label{left AxB-mod-on-tensor-2}(a \# b) (n \otimes y) = \sum_{(b)}( S^{-1}(b_{-1}y_{-1}) \rightharpoonup a)n \otimes b_0y_0.
 \end{equation}
This module, denoted also  by $N\#Y$, is called the {\it Takeuchi smash product} of $_AN$ and  $Y \in {_B^H\mathcal{M}}$.

\begin{rk} A right-sided version of Takeuchi smash product $B \# A$ is defined similarly for a right $H$-comodule algebra $B$ and a right $H$-module algebra $A$. There are also four types of Takeuchi smash product of modules similar to 
the module structures defined by \eqref{left AxB-mod-on-tensor}, \eqref{right AxB-mod-on-tensor-2}, \eqref{right AxB-mod-on-tensor} and \eqref{left AxB-mod-on-tensor-2}. The module structures given by  \eqref{right AxB-mod-on-tensor} and \eqref{left AxB-mod-on-tensor-2} seem less natural than the two given by  \eqref{left AxB-mod-on-tensor} and \eqref{right AxB-mod-on-tensor-2}, but they are dual to the two natural ones in the right-sided version in some sense (in particular when $H$ is finite-dimensional). 
\end{rk}

\subsection{Adjointness}
The following lemma is modified from the usual smash product case, the proof is routine.

\begin{lem} \label{Hom-Tensor-adj}
	Let $M \in {_{A\#H}\mathcal{M}}$, $X \in {}_B\mathcal{M}$ and $T \in {_{A\#B}\mathcal{M}}$.
	\begin{enumerate}
		\item There is a natural left $B$-module structure on $\Hom_A(M,T)$, given by
		\begin{equation} \label{B-mod-on-Hom}
			(b \rightharpoonup f)(m) = \sum\limits_{(b)}b_0f(S^{-1}b_{-1} \rightharpoonup m).
		\end{equation}
		\item The $\Hom$-Tensor adjoint isomorphism $\Hom(X, \Hom(M, T)) \stackrel{\cong}{\longrightarrow} \Hom(M \otimes X, T)$ restricts to a natural isomorphism
		$\Hom_B(X, \Hom_A(M,T)) \stackrel{\cong}{\longrightarrow} \Hom_{A\#B}(M\#X, T)$, that is, $\Hom_A(M, -): {_{A\#B}\mathcal{M}} \to {}_B\mathcal{M}$ and 
$M \# -: {}_B\mathcal{M} \to  {_{A\#B}\mathcal{M}}$ is a pair of adjoint functors.
	\end{enumerate}
\end{lem}

Then, for any $M \in {_{A\#H}\mathcal{M}}$ and $T \in  {_{A\#B}\mathcal{M}}$, there exists a natural $B$-module structure on the Ext-groups $\Ext^n_A(M, T)$. In fact, the $H$-module structure on the Ext-groups can be induced by any resolution of $M$ in ${_{A\#H}\mathcal{M}}$ such that each term is projective $A$-module, or any resolution of $T$ in ${_{A\#B}\mathcal{M}}$ of which all terms are injective $A$-modules. And these two induced $B$-module structures on $\Ext^n_{A\#B}(M, T)$ coincide, as shown in the following lemma.

\begin{lem}\label{H-mod str on Ext-group}
	Let $M \in {_{A\#H}\mathcal{M}}$ and $T \in {_{A\#B}\mathcal{M}}$.
	Let $P_{\bullet} \to M \to 0$ and $P'_{\bullet} \to M \to 0$ be exact sequences in $_{A\#H}\mathcal{M}$, where all $P_n$ and $P'_n$ are projective $A$-modules,
	and let $0 \to T \to I^{\bullet}$ and $0 \to T \to I'^{\bullet}$ be exact sequences in $_{A\#B}\mathcal{M}$, where all $I^n$ and $I'^n$ are injective $A$-modules.
	Then, for any $n \in \mathbb{N}$, as $B$-modules,
	$$\mathrm{H}^n(\Hom_{A}(P_{\bullet}, T)) \cong \mathrm{H}^n(\Hom_{A}(M, I^{\bullet})) \cong \mathrm{H}^n(\Hom_{A}(P'_{\bullet}, T)) \cong \mathrm{H}^n(\Hom_{A}(M, I'^{\bullet})).$$
\end{lem}
\begin{proof}
	Consider the double complex $\Hom_A(P_{\bullet}, I^{\bullet})$ of $B$-modules. Then
	$$\text{H}^*(\Hom_{A}(P_{\bullet}, T)) \cong \text{H}^*(\Hom_{A}(M, I^{\bullet}))$$
	as $B$-modules.
\end{proof}

For any $M_{A\#H}$, $X_B$ and $T_{A\#B}$, there is  right module version of Lemma \ref{Hom-Tensor-adj} and Lemma \ref{H-mod str on Ext-group}.

There is also a version of Lemma \ref{Hom-Tensor-adj} for any $N \in {\mathcal{M}_A}$, $Y \in {^H\mathcal{M}_{B}}$ and $T \in {\mathcal{M}_{A\#B}}$, which can be viewed as a dual version.
\begin{lem} \label{right-Hom-Tensor-adj}
	Let $N \in {\mathcal{M}_A}$, $Y \in {^H\mathcal{M}_{B}}$ and $T \in {\mathcal{M}_{A\#B}}$.
	\begin{enumerate}
		\item There is a right $A$-module structure on $\Hom_{B^{op}}(Y,T)$ given by
		$$(f a)(y) = \sum\limits_{(y)}f(y_0) (S^{-1}y_{-1} \rightharpoonup a).$$
		\item The $\Hom$-Tensor adjoint isomorphism $\Hom(N, \Hom(Y, T)) \stackrel{\cong}{\longrightarrow} \Hom(N \otimes Y, T)$ restricts to an isomorphism
		$\Hom_{A^{op}}(N, \Hom_{B^{op}}(Y,T)) \stackrel{\cong}{\longrightarrow} \Hom_{(A\#B)^{op}}(N\#Y, T).$
	\end{enumerate}
\end{lem}

\begin{lem} \label{relative-resolution}
	Let $X \in {_B^H\mathcal{M}}$ (resp. $^H\mathcal{M}_B$).
	\begin{enumerate}
		\item \cite[Proposition 4.1]{CG} $X$ viewed as left (resp. right) $B$-module is finitely generated if and only if there is a finite-dimensional left $H$-comodule $V$ and a surjective morphism $\pi: B \otimes V  \twoheadrightarrow X$ in $_B^H \mathcal{M}$ (resp.  $\pi: V \otimes B \twoheadrightarrow X$ in $^H\mathcal{M}_B$).
		\item Suppose  $X$ viewed as left (resp. right) $B$-module is pseudo-coherent. Then there is an exact sequence in $_B^H\mathcal{M}$ (resp. $^H\mathcal{M}_B$), where all the $V_n's$ are finite-dimensional left $H$-comodules,
		$$ \cdots \to B \otimes V_n \to \cdots \to B \otimes V_0 \to X \to 0 \,\, (\textrm{resp.}  \cdots \to V_n \otimes B \to \cdots \to V_0 \otimes B \to X \to 0).$$
	\end{enumerate}
\end{lem}

\begin{proof}
	(1) Let $\{x_1,\dots,x_m\}$ be a set generators of $_BX$.  By the Finiteness Theorem \cite[5.1.1]{Mo}, there exists a finite-dimensional  left $H$-comodule $V \subseteq X$ such that $\{x_1,\dots,x_m\} \subseteq V$. Hence a surjective morphism $\pi: B \otimes V \twoheadrightarrow X$ in $_B^H\mathcal{M}$ exists.
	
	(2) By (1), there exists a finite-dimensional  left $H$-comodule $V_0$ with a surjective morphism $\pi : B \otimes V_0 \twoheadrightarrow X$ in $_B^H\mathcal{M}$. By Schanuel's lemma, the kernel of $\pi$ is also a pseudo-coherent $B$-module. Then the conclusion follows by induction.
\end{proof}

\subsection{Comodule structure on the B-morphism space of (H, B)-Hopf modules} The material in this subsection is essentially modified from \cite{Ul}, especially \cite[Lemma 2.2]{Ul}.

For any $X, Y \in {_B^{H}\mathcal{M}}$, we try to endow $\Hom_B(X,Y)$ with a left $H$-comodule structure, say, when $_BX$ is finitely generated. So, firstly we have to define a map $\Hom_B(X, Y) \to H \otimes \Hom_B(X, Y).$
We use the natural embedding $$H \otimes \Hom_B(X,Y) \subseteq \Hom_B(X, H \otimes Y), \, h \otimes f \mapsto \big(x \mapsto  h \otimes f(x)\big).$$

Let $\rho=\rho_Y: \Hom_B(X, Y) \to \Hom_B(X, H \otimes .Y)$ be the  morphism given by
\begin{equation} \label{H-comod-on-HOM}
	\rho(f)(x) = \sum_{(x),(f(x_0))} (Sx_{-1}) f(x_0)_{-1} \otimes f(x_0)_0
\end{equation}
where $f \in \Hom_B(X, Y)$ and $x \in X$.
Let
$$\HOM_B(X,Y) := \{ f \in \Hom_B(X,Y) \mid \rho(f) \in H \otimes \Hom_B(X,Y) \subseteq \Hom_B(X, H \otimes Y) \}.$$
Then, $\HOM_B(X,Y)$ is a left $H$-comodule. A proof of this is added in the following for the convenience.

\begin{lem} \label{H-comod-str-on-Hom}
Let $X, Y \in {}_B^H\mathcal{M}$.
	\begin{enumerate}
		\item $\HOM_B(X, Y)$ has a left $H$-comodule structure as given by \eqref{H-comod-on-HOM}.
		\item If $_BX$ is finitely generated, then $\Hom_B(X, Y) = \HOM_B(X, Y)$, and so it is a left $H$-comodule (see also \cite[Proposition 4.2]{CG}).
	\end{enumerate}
\end{lem}
\begin{proof} (1)
	For any $f \in \HOM_B(X,Y)$, suppose $\rho(f) = \sum_i h^i \otimes f^i \in H \otimes \Hom_B(X,Y)$, where $\{h^i\} \subseteq H$ and $\{f^i\} \subseteq \Hom_B(X,Y)$  are $k$-linear independent respectively.
	Then $$\rho(f)(x)=\sum_i h^i \otimes f^i(x) = \sum_{(x),(f(x_0))} (Sx_{-1}) f(x_0)_{-1} \otimes f(x_0)_0, \, \textrm{ and }$$
	$$\sum_{(x)} Sx_{-1} \otimes \rho(f)(x_0) = \sum_{i, (x)} Sx_{-1} \otimes h^i \otimes f^i(x_0)= \sum_{ (x), (f(x_0)} Sx_{-2} \otimes (Sx_{-1})f(x_0)_{-1} \otimes f(x_0)_0.$$
	By acting by $\rho$ on the last component of both sides in the last equality, and switching the first two components obtained, then
	\begin{equation}\label{equality-for-next}
		\sum_i \sum_{(x), (f^i(x_0))} h^i \otimes (Sx_{-1})f^i(x_0)_{-1} \otimes f^i(x_0)_0
		=\sum_{(x), (f(x_0))} (Sx_{-1})f(x_0)_{-2} \otimes (Sx_{-2})f(x_0)_{-1} \otimes f(x_0)_0.
	\end{equation}
	Consider the images of $\sum_{i, (h^i)} h^i_1 \otimes h^i_2 \otimes f^i \in H \otimes H \otimes \Hom_B(X,Y)$ and $\sum_i h^i \otimes \rho(f^i) \in H \otimes \Hom_B(X, H \otimes Y)$ in $\Hom_B(X, H \otimes H \otimes Y)$. By acting on $x \in X$, then, on one hand,
$$(\Delta \otimes 1)\rho(f)(x) = (\Delta \otimes 1)\sum_i(h^i \otimes f^i(x))
= \sum_i\sum_{(h^i)} h^i_1 \otimes h^i_2 \otimes f^i(x).$$
On the other hand,
\begin{align*}
  &(\Delta \otimes 1)\rho(f)(x)=\sum_{(x), (f(x_0))} (Sx_{-1})f(x_0)_{-2} \otimes (Sx_{-2})f(x_0)_{-1} \otimes f(x_0)_0 \\
		\stackrel{\eqref{equality-for-next}}{=} &\sum_i \sum_{(x), (f^i(x_0))} h^i \otimes (Sx_{-1})f^i(x_0)_{-1} \otimes f^i(x_0)_0 = \sum_i h^i \otimes \rho(f^i)(x) \\
=& \sum_i (1_H \otimes \rho) (h^i \otimes f^i)(x)\\
		= &(1_H \otimes \rho) \rho(f)(x).
	\end{align*}
	Now it is easy to see that $f^i \in \HOM_B(X,Y)$ for all $i$, and $\HOM_B(X,Y)$ is a left $H$-comodule.
	
	(2)
	By Lemma \ref{relative-resolution}, there exists a finite-dimensional left $H$-subcomodule $V$ of $X$, and a surjective morphism $\pi: B \otimes V \to X$ in ${_B^{H}\mathcal{M}}$. So, $\Hom_B(X,Y) \stackrel{\pi^*}{\longrightarrow} \Hom_B(B \otimes V, Y)$ is injective,
	and the following diagram is commutative.
	$$\xymatrix{
		\Hom_B(X, Y) \ar[d]^{\rho_X} \ar@{>->}[r]^{\pi^*} & \Hom_B(B \otimes V, Y) \ar[d]^{\rho_{B \otimes V}} \\
		\Hom_B(X, H \otimes Y) \ar@{>->}[r]^{\pi^*} & \Hom_B(B \otimes V, H \otimes Y)\\
		H \otimes \Hom_B(X,Y) \ar[r]^{1_H \otimes \pi^*} \ar@{>->}[u] & H \otimes \Hom_B(B \otimes V, Y)  \ar[u]^{=}
	}
	$$
	
	Now for any $f \in \Hom_B(X,Y)$, suppose $\rho_{B \otimes V}(f \pi) = \sum h^i \otimes g^i$, where $\{h^i\} \in H$ and $\{g^i\} \in \Hom_B(B \otimes V,Y)$
	are $k$-linear independent respectively.
	Then, for any $x' \in \ker \pi$, $$\sum h^i \otimes g^i(x') = \rho_X(f)(\pi(x')) = 0.$$ It follows that $g^i(x') = 0$ for any $i$, and so $g^i = f^i\pi$ for some $f^i \in \Hom_B(X,Y)$.
	Since $\Hom_B(X,Y) \stackrel{\pi^*}{\longrightarrow} \Hom_B(B \otimes V, Y)$ is injective, $\rho_X(f)=\sum h^i \otimes f^i \in H \otimes\Hom_B(X,Y)$ and $\Hom_B(X,Y) = \HOM_B(X,Y)$.
\end{proof}

If fact, $\HOM_B(-,-):  ({_B^{H}\mathcal{M}})^{op} \times {_B^{H}\mathcal{M}}  \to  {_{H}\mathcal{M}}$ is  a functor.
For any $X, X' \in {_B^{H}\mathcal{M}}$ with $_BX$ pseudo-coherent, by Lemma \ref{H-comod-str-on-Hom}, there is a natural $H$-comodule structure on $\Ext^*_B(X,X')$. The $H$-comodule structure on the Ext-groups is independent of the choices of the resolutions of $X$ in ${_B^{H}\mathcal{M}}$ such that each term is finitely generated projective as $B$-module and the resolutions of $X'$ in ${_B^{H}\mathcal{M}}$ of which all terms are injective as $B$-modules.

\begin{lem}\label{H-comod str on Ext-group}
	Let $X, X' \in {_B^{H}\mathcal{M}}$. Suppose that $P_{\bullet} \to X \to 0$, $P'_{\bullet} \to X \to 0$, $0 \to X' \to I^{\bullet}$ and $0 \to X' \to I'^{\bullet}$ are exact sequences
	in $_B^H\mathcal{M}$, where all $P_n$ and $P'_n$ are finitely generated projective as $B$-module, and all $I^n$ and $I'^n$ are injective as $B$-module. Then, for any $n \in \mathbb{N}$, as $H$-comodules,
	$$\mathrm{H}^n(\Hom_{B}(P_{\bullet}, X')) \cong \mathrm{H}^n(\Hom_{B}(X, I^{\bullet})) \cong \mathrm{H}^n(\Hom_{B}(P'_{\bullet}, X')) \cong \mathrm{H}^n(\Hom_{B}(X, I'^{\bullet})).$$
\end{lem}
\begin{proof}
	Similar to the proof of Lemma \ref{H-mod str on Ext-group}.
\end{proof}

\section{Some homological properties of Takeuchi smash products}

In this section, we consider some (homological) properties of Takeuchi smash products of modules, especially, the decompositions of the Tor-groups and Ext-groups of Takeuchi smash products of modules.





Let $A$ be a left $H$-module algebra and $B$ be a left $H$-comodule algebra.
With the natural module structure, $_A(A\#B)$ and $(A\#B)_B$ are obviously free modules.

\begin{lem}\label{proj-Takeuchi-smash}
 Both $(A\#B)_A$ and ${}_B(A\#B)$ are free modules.
\end{lem}
\begin{proof}
	The conclusion follows from the $B$-$A$-bimodule isomorphism $\varphi: {}_B(A\#B)_A \to {}_B B \otimes A_A$ defined by $\varphi(a \# b) = \sum_{(b)} b_0 \otimes (S^{-1}b_{-1} \rightharpoonup a)$.
\end{proof}

\begin{prop} \label{Ext-spec-seq-of-smash-prod-mod}
	For any $M \in {_{A\#H}\mathcal{M}}$, $X \in {_{B}\mathcal{M}}$ and $T \in {_{A\#B}\mathcal{M}}$,
	there is a convergent spectral sequence
	$$\Ext_B^p(X, \Ext_A^q(M,T)) \Rightarrow \Ext^{p+q}_{A\#B}(M\#X, T).$$
\end{prop}
\begin{proof}
	Let $P_{\bullet}$ be a projective resolution of $_BX$, and let $I^{\bullet}$ be an injective resolution of $_{A\#B}T$. Since $A\#B$ is a flat right $A$-module by Lemma \ref{proj-Takeuchi-smash}, each $I^i$ is injective as an $A$-module.
	Since $\Hom_A(M, I^i)$ is an injective $B$-module by Lemma \ref{Hom-Tensor-adj}, then the double complex $$\Hom_{A\#B}(M\#P_{\bullet}, I^{\bullet}) \cong \Hom_{B}(P_{\bullet}, \Hom_A(M, I^{\bullet}))$$
	yields the spectral sequence.
\end{proof}

\begin{cor}\label{proj-mod of smash product}
	For any $X \in {_{B}\mathcal{M}}$ and $M \in {_{A\#H}\mathcal{M}}$, $\pd_{A\#B}(M\#X) \leq \pd_A M + \pd_B X$. In particular, if both $_AM$ and $_BX$ are projective modules, then $M\#X$ is a projective $A\#B$-module.
\end{cor}

There is also right module versions of 
Proposition \ref{Ext-spec-seq-of-smash-prod-mod} and Corollary \ref{proj-mod of smash product}.


\begin{prop} \label{right-case-spec-seq}
	For any $N \in {\mathcal{M}_A}$, $Y \in {^H\mathcal{M}_B}$ and $T \in {\mathcal{M}_{A\#B}}$,
	there is a convergent spectral sequence
	$$\Ext_{A^{op}}^p(N, \Ext_{B^{op}}^q(Y,T)) \Rightarrow \Ext^{p+q}_{(A\#B)^{op}}(N\#Y, T).$$
\end{prop}

\begin{cor}\label{proj-mod of smash product'}
	For any $N \in {\mathcal{M}_A}$, $Y \in {^H\mathcal{M}_{B}}$, $\pd_{(A\#B)^{op}}(N\#Y) \leq \pd_{A^{op}} N + \pd_{B^{op}} Y$. In particular, if both $N_A$ and $Y_B$ are projective modules, then $N\#Y$ is a projective right $A\#B$-module.
\end{cor}

\begin{prop}\label{decom-Tor-group}
	Suppose $M \in {_{A\#H}\mathcal{M}}$, $X \in {_{B}\mathcal{M}}$, $N \in \mathcal{M}_A$ and $Y \in {^{H}\mathcal{M}_B}$. Then,
	\begin{enumerate}
		\item $(N\#Y) \otimes_{A\#B} (M\#X) \cong (N \otimes_A M) \otimes (Y \otimes_B X)$.
		\item $\Tor^{A\#B}_n(N\#Y, M\#X) \cong \bigoplus\limits_{p+q=n} \Tor^{A}_p(N, M) \otimes \Tor^{B}_q(Y, X)$.
	\end{enumerate}
\end{prop}
\begin{proof}
	(1) It is easy to check that the map $\phi : (N\#Y) \otimes_{A\#B} (M\#X) \to (N \otimes_A M) \otimes (Y \otimes_B X)$
	$$(n\#y) \otimes_{A\#B} (m\#x) \mapsto \sum_{(y)} (n \otimes_A y_{-1} \rightharpoonup m) \otimes (y_0 \otimes_B x)$$
	is well-defined. The map $\phi$ is an isomorphism with the inverse $\phi^{-1}$  given by
	$$ (n \otimes_A m) \otimes (y \otimes_B x) \mapsto \sum_{(y)} (n\# y_0) \otimes_{A\#B} (S^{-1}y_{-1} \rightharpoonup m \, \# x).$$
	
	(2) Let $P_{\bullet}$ be a projective resolution of $_{A\#H}M$, and $Q_{\bullet}$ be a projective resolution of $_BX$. By Corollary \ref{proj-mod of smash product},
	the total complex $\mathrm{Tot}(P_{\bullet} \# Q_{\bullet})$ is a projective resolution of $_{A\#B}(M\#X)$. It follows from  (1) that we have an isomorphism of complexes
	$$(N\#Y) \otimes_{A\#B} \mathrm{Tot}(P_{\bullet} \# Q_{\bullet}) \cong \mathrm{Tot}((N \otimes_A P_{\bullet}) \otimes (Y \otimes_B Q_{\bullet})).$$
	The conclusion follows by taking the homology.
\end{proof}


\subsection{Ext-groups of Takeuchi smash product modules}

 As a special case in Lemma \ref{Hom-Tensor-adj} (1) where $B=H$, $\Hom_A(M,M')$ is a left $H$-module for $M, M' \in {_{A\#H}\mathcal{M}}$,  with the $H$-action given by
		$$(h \rightharpoonup f)(m) = \sum h_2 \rightharpoonup f(S^{-1}h_1 \rightharpoonup m).$$

\begin{lem}\label{an-iso-for-Ext}
	Let $M, M' \in {_{A\#H}\mathcal{M}}$, $X \in {_B^{H}\mathcal{M}}$, and $X' \in {_{B}\mathcal{M}}$.
%
There is a morphism
      \begin{align*}
\Phi: \Hom_A(M, M') \otimes \Hom_B(X, X') & \longrightarrow  \Hom_{B}(X, \Hom_A(M,M') \otimes X') \\
 f \otimes g & \mapsto [f, g]: x \mapsto (x_{-1} \rightharpoonup f) \otimes g(x_0)
      \end{align*}
  where $\Hom_A(M,M')\otimes X'$ is viewed as a left $B$-module via the algebra map
		$\rho: B \to H \otimes B$.

 If  ${}_BX$ is finitely presented, then $\Phi$ is an isomorphism.
\end{lem}
\begin{proof} For any $k$-vector space $U$, the natural map $U \otimes \Hom_B(X, X') \to \Hom_B(X, U \otimes .X'), u \otimes g\mapsto \big([u, g], x \mapsto u \otimes g(x)\big)$ is an isomorphism provided ${}_BX$ is finitely presented.

If $U$ is a left $H$-module, then
$\Hom_B(X, U \otimes .X') \cong \Hom_B(X, U \otimes X'), \, F \mapsto \big({\bar F}: x \mapsto \sum x_{-1} \rightharpoonup F(x_0)\big)$, where $U \otimes X'$ in the second $\Hom$ is viewed as left $B$-module via the algebra map $\rho: B \to H \otimes B$. The inverse map is $G \mapsto  \big({\tilde G}: x \mapsto  \sum Sx_{-1} \rightharpoonup G(x_0)\big)$.

The composition of above two maps with $U=\Hom_A(M, M')$ is the given morphism $\Phi$.
\end{proof}

\begin{prop}\label{decom-Ext-group}
	Let $M, M' \in {_{A\#H}\mathcal{M}}$, $X \in {_B^{H}\mathcal{M}}$, and $X' \in {_{B}\mathcal{M}}$.
	\begin{enumerate}
\item There is a natural morphism
		$$\psi: \Hom_A(M, M') \otimes \Hom_B(X, X') \longrightarrow \Hom_{A\#B}(M\#X, M'\# X')$$
		$$f \otimes g \mapsto \big(m\#x \mapsto \sum_{(x)} (x_{-1}\rightharpoonup f)(m) \# g(x_0)\big).$$
If $M$ viewed as $A$-module and $X$ viewed as $B$-module are finitely presented, then $\psi$ is an isomorphism.
		\item Suppose that $_AM$ and $_BX$ are pseudo-coherent modules. Then for any $n \in \mathbb{N}$
		$$ \Ext^n_{A\#B}(M\#X, M'\# X') \cong \bigoplus_{p+q=n} \Ext^p_{A}(M, M') \otimes \Ext^q_{B}(X, X').$$
	\end{enumerate}
\end{prop}

\begin{proof}
(1) As noted before, $\Hom_A(M,M')\otimes X'$ is viewed as a left $B$-module  by the algebra morphism
		$\rho: B \to H \otimes B$, that is,  $b \rightharpoonup (f \otimes y) := \sum (b_{-1} \rightharpoonup f) \otimes b_0y.$
 By \eqref{B-mod-on-Hom}, $\Hom_A(M, M'\#X')$ is a left $B$-module.
 Then the natural  morphism $\sigma: \Hom_A(M,M') \otimes X' \to \Hom_A(M,M'\#X'), f \otimes y \mapsto \big(m \mapsto f(m) \# y \big)$ is a $B$-module morphism. If $_AM$ is finitely presented as an $A$-module, then $\sigma$ is  an  isomorphism.

The composition $\psi$ in the following commutative diagram is the morphism as claimed.
$$
\xymatrix{
  \Hom_A(M,M') \otimes \Hom_B(X,X')  \ar@{.>}[d]^{\psi}\quad  \ar[r]^{\text{Lemma \ref{an-iso-for-Ext}}} & \quad \Hom_{B}(X, \Hom_A(M,M')\#X') \ar[d]^{\sigma_*}  \\
  \Hom_{A\#B}(M\#X, M'\#X')
                & \Hom_{B}(X, \Hom_A(M,M'\#X')) \ar[l]^{\cong}_{\text{Lemma \ref{Hom-Tensor-adj}}}
                 }
$$
If both $_AM$ and $_BX$ are finitely presented, then $\psi$ is an isomorphism.

		
   (2) By Lemma \ref{relative-resolution}(2), there is an exact sequence
		$ \cdots \to Q_n \to \cdots \to Q_0 \to X \to 0$ in $_B^H\mathcal{M}$ with $Q_n=B \otimes V_n$ where $V_n$ is some finite-dimensional $H$-comodule.
		Let $P_{\bullet}$ be a projective resolution of $_{A\#H}M$. It follows from Corollary \ref{proj-mod of smash product} that $\mathrm{Tot}(P_{\bullet} \# Q_{\bullet})$ is a projective resolution of $_{A\#B}(M\#X)$.
Since ${_AM}$ is pseudo-coherent, there is a finitely generated $A$-projective resolution ${\widetilde P}_{\bullet} \to M \to 0$ of $M$. Then, $ P_{\bullet}$ and ${\widetilde P}_{\bullet}$ are homotopic equivalent, and the following is a commutative diagram of complexes
$$
\xymatrix{
  \Hom_A(P_{\bullet},M')\otimes X' \ar[r]^{\sigma_\bullet}  \ar[d]^{\text{homo. equ.}} & \Hom_A(P_{\bullet}, M'\# X') \ar[d]^{\text{homo. equ.}}   \\
 \Hom_A({\widetilde P}_{\bullet},M')\otimes X'  \ar[r]^{\sigma_\bullet}_{\cong}
                & \Hom_A({\widetilde P}_{\bullet},M'\# X')
                 }
$$
It follows that $ \Hom_A(P_{\bullet},M')\otimes X' \stackrel{\sigma_\bullet}{\longrightarrow} \Hom_A(P_{\bullet}, M'\# X')$ is a quasi-isomorphism of $B$-module complexes. Hence
$$\mathrm{Tot}(\Hom_{B}(Q_{\bullet}, \Hom_A(P_{\bullet},M')\otimes X'))\xrightarrow{\sim} \mathrm{Tot}(\Hom_{B}(Q_{\bullet}, \Hom_A(P_{\bullet},M'\#X')))$$
where $\xrightarrow{\sim}$ means quasi-isomorphism.
		\begin{align*}
		 & \Hom_{A\#B}(\mathrm{Tot}(P_{\bullet} \# Q_{\bullet}), M'\#X') \\
		\cong & \mathrm{Tot}(\Hom_{B}(Q_{\bullet}, \Hom_A(P_{\bullet},M'\#X')))  \quad \text{ (by Lemma \ref{Hom-Tensor-adj})} \\
		\xleftarrow{\sim} & \mathrm{Tot}(\Hom_{B}(Q_{\bullet}, \Hom_A(P_{\bullet},M')\otimes X'))  \\
		\cong & \mathrm{Tot}(\Hom_A(P_{\bullet},M') \otimes \Hom_{B}(Q_{\bullet}, X')) \quad \text{ (as $Q_n=B \otimes V_n$ is finitely generated projective)}.
		\end{align*}
	 The conclusion follows by taking cohomologies.
\end{proof}

\begin{cor}
	Suppose further both $A$ and $B$ are semi-primary rings over a perfect field $k$. If the Jacobson radical $J_A$ of $A$ is an $H$-submodule, and the Jacobson radical $J_B$ of $B$ is an $H$-subcomodule, then $\gld A\#B = \gld A + \gld B$.
\end{cor}
\begin{proof}
	By assumption, $J := A \# J_B + J_A \# B$ is a nilpotent ideal of $A\#B$. Since $k$ is a perfect field, $(A\#B)/ J \cong A/J_A \# B/J_B$ is a semisimple ring. Hence $A\#B$ is also a semi-primary ring. By \cite[Corollary 12]{Aus}, $\gld A\#B = \sup\{ n \mid \Ext^n_{A\#B}((A\#B)/J, (A\#B)/J) \}$. The conclusion follows from Proposition \ref{decom-Ext-group}.
\end{proof}

Also we have the right version of Proposition \ref{decom-Ext-group}.

For any $M \in {\mathcal{M}}_{A\#H}$, $M$ is sometimes viewed as a left $H$-module via $h \rightharpoonup m := m(1 \# Sh)$. Thus $h \rightharpoonup (ma) = \sum\limits_{(h)} (h_1 \rightharpoonup m)(h_2 \rightharpoonup a)$. For any $M, N \in \mathcal{M}_{A\#H}$, $\Hom_{A^{op}}(M,N) \in {_H\mathcal{M}}$, with the $H$-action given by
\begin{equation}\label{H-mod on right Hom set}
(h \rightharpoonup f)(m) = \sum_{(h)} h_1 \rightharpoonup f(Sh_2 \rightharpoonup m)
\end{equation}
where $m \in M, h \in H, f \in \Hom_{A^{op}}(M,N)$.

\begin{prop}\label{decom-Ext-group'}
	Let $M \in {\mathcal{M}}_{A\#H}$, $N \in {\mathcal{M}}_{A}$, and $X, Y \in {^{H}\mathcal{M}_B}$. Suppose $X_B$ is finitely generated. Then,
	\begin{enumerate}
		\item $\Hom_{B^{op}}(X,Y) \in {^H\mathcal{M}}$, with the $H$-coaction given by
		\begin{eqnarray}\label{H-comod on right Hom set}
			\rho(f)(x) = \sum_{(x)} f(x_0)_{-1}S^{-1}x_{-1} \otimes f(x_0)_0
		\end{eqnarray}
		for any $x \in X$ and $f \in \Hom_A(M,N)$.
		\item There exists a natural morphism
		$$\phi: \Hom_{A^{op}}(M,N) \otimes \Hom_{B^{op}}(X,Y) \longrightarrow \Hom_{(A\#B)^{op}}(M\#X, N\#Y)$$
		$$f \otimes g \mapsto \big(m\#x \mapsto \sum_{(g)} f(m (1 \# Sg_{-1})) \# g_0(x)\big)$$
		for any $m \in M, x \in X$. If $M_A$ and $X_B$ are finitely presented modules, then $\phi$ is an isomorphism.
		\item Suppose that $M_A$ and $X_B$ are pseudo-coherent modules. Then for any $n \in \mathbb{N}$
		$$ \Ext^n_{(A\#B)^{op}}(M\#X, N\#Y) \cong \bigoplus_{p+q=n} \Ext^p_{A^{op}}(M, N) \otimes \Ext^q_{B^{op}}(X, Y).$$
	\end{enumerate}
\end{prop}


\section{Yoneda Ext-algebras of Takeuchi smash products}

\subsection{Smash products of differential graded algebras}


A graded algebra $\Lambda = \bigoplus _{n \in \Z} \Lambda_n$ is called a graded $H$-module algebra, if $\Lambda$ is an $H$-module algebra, and $h \rightharpoonup \Lambda_n \subseteq \Lambda_n$ for all $h \in H$ and $n \in \Z$. For any non-zero homogeneous element $\lambda \in \Lambda_n$, $|\lambda| := n$ is called the degree of $\lambda$.

Let $\Lambda$ be a differential graded algebra (or dga, for short), that is, $\Lambda = \bigoplus \Lambda_n$ is a graded algebra with a graded map $d: \Lambda \to \Lambda$ of degree $1$, such that $d^2 = 0$ and the Leibnitz rule holds:
$$d(\lambda \cdot \lambda') = d(\lambda) \cdot \lambda' + (-1)^{|\lambda|} \lambda \cdot d(\lambda'),  \, \text{ for any  homogeneous element } \lambda \text{ and } \lambda' \in \Lambda.$$
$\Lambda$ is called a {\it differential graded left $H$-module algebra}, if $\Lambda = \bigoplus \Lambda_n$ is a graded $H$-module algebra and the differential $d$ is an $H$-morphism, that is, $d(h \rightharpoonup \lambda) = h \rightharpoonup d(\lambda)$ for all $h \in H$ and $\lambda \in \Lambda$.

{\it Differential graded left $H$-comodule algebras} are defined analogously.

Suppose that $\Lambda$ is a differential graded left $H$-module algebra, and $\Gamma$ is a differential graded left $H$-comodule algebra. Then
$\Lambda\#\Gamma$ is also a (differential) graded algebra with the  multiplication defined by
$$(\lambda\#\gamma)(\lambda'\#\gamma') = \sum_{(\gamma)} (-1)^{|\gamma||\lambda'|} \lambda(\gamma_{-1} \rightharpoonup \lambda') \# \gamma_0 \gamma'$$
and the differential $d_{\Lambda \# \Gamma}$ defined by
$$d_{\Lambda \# \Gamma}(\lambda\#\gamma) = d_{\Lambda}(\lambda) \# \gamma + (-1)^{|\lambda|} \lambda \# d_{\Gamma}(\gamma)$$
for any homogeneous elements $\lambda, \lambda' \in \Lambda$, and $\gamma, \gamma' \in \Gamma$.

Let $R$ be a ring,  $P_{\bullet}$ and $Q_{\bullet}$ be two $R$-module complexes. Consider the graded space $\Hom^{\bullet}_R(P_{\bullet}, Q_{\bullet})$ with
$$ \Hom^n_R(P_{\bullet}, Q_{\bullet}) = \prod_i \Hom_R(P_i, Q_{i-n}). $$
The differentials on $P_{\bullet}$ and $Q_{\bullet}$ which are of degree $-1$ give rise to a differential $d_{\Hom}$ of degree $1$ on $\Hom^{\bullet}_R(P_{\bullet}, Q_{\bullet})$, where
$d_{\Hom}(f) = d^Q \circ f - (-1)^{|f|}f \circ d^P = (d^Q_{i-|f|} \circ f_i - (-1)^{|f|}f_{i-1} \circ d^P_i)$
for $f =(f_i)_i$, and $f_i \in \Hom_R(P_i, Q_{i-|f|})$.

In fact, $\End^{\bullet}_R(P_{\bullet}) = \Hom^{\bullet}_R(P_{\bullet},P_{\bullet})$ is a dga, as  $d_{\Hom}(fg) = d_{\Hom}(f)g + (-1)^{|f|} fd_{\Hom}(g)$  for any $f,g \in \End^{\bullet}_R(P_{\bullet})$. So, the cohomology group $\mathrm{H}^*(\End^{\bullet}_R(P_{\bullet}))$ has a graded algebra structure.

\begin{prop}\label{Hom-mod-comod-dga}
	Let $M_{\bullet}$ be a complex in $\mathcal{M}_{A\#H}$, and $X_{\bullet}$ be a complex in $^H\mathcal{M}_B$. Suppose that each term $X_i$ viewed as a $B^{op}$-module is finitely generated. Then
	\begin{enumerate}
		\item $\End_{A^{op}}^{\bullet}(M_{\bullet})$ is a differential graded left $H$-module algebra, with the module structure given in \eqref{H-mod on right Hom set}.
		\item $\End_{B^{op}}^{\bullet}(X_{\bullet})$ is a differential graded left $H$-comodule algebra, with the comodule structure given in \eqref{H-comod on right Hom set}.
		\item There is a differential graded algebra morphism
		$$\varphi: \End_{A^{op}}^{\bullet}(M_{\bullet}) \# \End_{B^{op}}^{\bullet}(X_{\bullet}) \to \End_{(A\#B)^{op}}^{\bullet}(\mathrm{Tot}(M_{\bullet} \# X_{\bullet}))$$
		$$f \# g \mapsto \big( m\#x \mapsto (-1)^{|g| |m|} \sum\limits_{(g)} f(g_{-1} \rightharpoonup m) \# g_0(x)\big),$$
		which is induced by Proposition \ref{decom-Ext-group'} (2).
	\end{enumerate}
\end{prop}
\begin{proof}
	(1) For any $h \in H, f, f' \in \End_{A^{op}}^{\bullet}(M_{\bullet})$ and $m \in M_i$,
	\begin{align*}
		&((h \rightharpoonup (ff'))(m)  = \sum\limits_{(h)}h_1 \rightharpoonup\big(ff'(Sh_2 \rightharpoonup m)\big) \\
 = & \sum\limits_{(h)} h_1 \rightharpoonup f\big((Sh_2) h_3\rightharpoonup f'(Sh_4 \rightharpoonup m)\big) = \sum\limits_{(h)} h_1 \rightharpoonup f\big(Sh_2 \rightharpoonup \big((h_3\rightharpoonup f')(m)\big)\big) \\
		 = & \sum\limits_{(h)} (h_1 \rightharpoonup f)(h_2 \rightharpoonup f')(m).
	\end{align*}
It follows that $h \rightharpoonup (ff')=\sum\limits_{(h)} (h_1 \rightharpoonup f)(h_2 \rightharpoonup f')$.

Since $(h \rightharpoonup \id_{M_{\bullet}})(m) = \sum\limits_{(h)}h_1 \rightharpoonup \id_{M_{\bullet}}(Sh_2 \rightharpoonup m) 
= \varepsilon(h) (m)$, $(h \rightharpoonup \id_{M_{\bullet}})=\varepsilon(h) \id_{M_{\bullet}}$.

Hence $\End_{A^{op}}^{\bullet}(M_{\bullet})$ is a graded $H$-module algebra.
In fact, the differential $d_{\Hom}$ is an $H$-morphism.
	\begin{align*}
		& (d_{\Hom}(h \rightharpoonup f))(m)  = d^M \circ (h \rightharpoonup f)(m) - (-1)^{|f|}(h \rightharpoonup f) \circ d^M(m) \\
		 = &\sum_{(h)} d^M(h_1 \rightharpoonup f(Sh_2 \rightharpoonup m)) - (-1)^{|f|} h_1 \rightharpoonup f(Sh_2 \rightharpoonup d^M(m)) \\
		= &\sum_{(h)} h_1 \rightharpoonup (d^M  f (Sh_2 \rightharpoonup m)) - (-1)^{|f|}h_1 \rightharpoonup (f  d^M(Sh_2 \rightharpoonup m)) \\
		= & ( h \rightharpoonup d_{\Hom}(f))(m).
	\end{align*}
	It follows that $\End_{A^{op}}^{\bullet}(M_{\bullet})$ is a differential graded left $H$-module algebra.
	
	(2) For any $g,g' \in \End_{B^{op}}^{\bullet}(X_{\bullet})$ and  $x \in X_{\bullet}$, by \eqref{H-comod on right Hom set}, $$\rho(g')(x) = \sum g'_{-1} \otimes g'_0(x) = \sum\limits_{(x)}  g'(x_0)_{-1}S^{-1}x_{-1} \otimes g'(x_0)_0,$$
	and $$\rho(g)\big(g'(x_0)_0\big)= \sum g_{-1}\otimes g_0\big(g'(x_0)_0\big)=\sum\limits_{(g'(x_0)_0)}  g(g'(x_0)_0)_{-1}(S^{-1}g'(x_0)_{-1}) \otimes g(g'(x_0)_0)_0.$$
Then
	\begin{align*}
		& \sum\limits_{(g),(g')} g_{-1}g'_{-1} \otimes g_0g'_0(x) = \sum\limits_{(g),(x),g'(x_0)} g_{-1} g'(x_0)_{-1}S^{-1}x_{-1} \otimes g_0(g'(x_0)_0)\\
		 = & \sum\limits_{(x),(g'(x_0)),(g(g'(x_0)_0))} g(g'(x_0)_0)_{-1}(S^{-1}g'(x_0)_{-1}) g'(x_0)_{-2}S^{-1}x_{-1} \otimes g(g'(x_0)_0)_0 \\
		 = & \sum\limits_{(x),(gg'(x_0))} g(g'(x_0))_{-1}S^{-1}x_{-1} \otimes g(g'(x_0))_0 \\
		 = & \sum\limits_{(gg')} (gg')_{-1} \otimes (gg')_0(x).
	\end{align*}
	It follows that  $\rho(gg') = \rho(g) \rho(g')$.

Obviously, $(\varepsilon \otimes 1)\rho(g) = g$. Hence $\End_{B^{op}}^{\bullet}(X_{\bullet})$ is a graded left $H$-comodule algebra. The differential $d_{\Hom}$ is also an $H$-comodule morphism as shown next. Since
	\begin{align*}
		&\rho(d^X \circ g)(x)  = \sum_{(x)} (d^X \circ g)(x_0)_{-1}S^{-1}x_{-1} \otimes (d^X \circ g)(x_0)_0 \\
		 = & \sum_{(x)} g(x_0)_{-1}S^{-1}x_{-1} \otimes d^X (g(x_0)_0) \\
		 = & \sum_{(x)} g_{-1} \otimes d^X  (g_0(x)),
	\end{align*}
it follows that
\begin{align*}
		&\rho(d_{\Hom} (g))(x)  = \sum_{(x)} d_{\Hom} (g)(x_0)_{-1}S^{-1}x_{-1} \otimes d_{\Hom} ( g)(x_0)_0 \\
		 = & \sum_{(x)} g(x_0)_{-1}S^{-1}x_{-1} \otimes d_{\Hom} (g(x_0)_0) \\
		 = & \sum_{(x)} g_{-1} \otimes d_{\Hom} (g_0)(x).
	\end{align*}
Hence $\rho(d_{\Hom} (g))= \sum_{(g)} g_{-1} \otimes d_{\Hom} (g_0)$ and $\End_{B^{op}}^{\bullet}(X_{\bullet})$ is a differential graded left $H$-comodule algebra.
	
	(3) For any $f,f' \in \End_{A^{op}}^{\bullet}(M_{\bullet}), g,g' \in \End_{B^{op}}^{\bullet}(X_{\bullet}), m \in M_i, x \in X_l$,
	\begin{align*}
	& \varphi\big((f\#g) (f'\#g')\big)(m\#x)  = (-1)^{|g||f'|}\varphi\big(\sum_{(g)} f(g_{-1} \rightharpoonup f') \# g_0g'\big)(m\#x) \\
		 = &(-1)^{|g||f'|} (-1)^{|m|(|g| + |g'|)} \sum_{(g),(g')} f(g_{-2} \rightharpoonup f')(g_{-1}g'_{-1} \rightharpoonup m) \# g_0g'_0(x) \\
		 = &(-1)^{|g||f'|}(-1)^{|m|(|g| + |g'|)} \sum_{(g),(g')} f\big(g_{-1} \rightharpoonup f'(g'_{-1} \rightharpoonup m)\big) \# g_0g'_0(x) \\
		 = &(-1)^{|m||g'|} \sum_{(g')} \phi(f\#g)\big(f'(g'_{-1} \rightharpoonup m) \# g'_0(x)\big) \\
		 = & \varphi(f\#g)\varphi(f'\#g')(m\#x).
	\end{align*}
Hence  $\varphi\big((f\#g) (f'\#g')\big) = \varphi(f\#g)\varphi(f'\#g').$

	Clearly $\varphi(\id_{M_{\bullet}} \# \id_{X_{\bullet}}) = \id_{\mathrm{Tot}(M_{\bullet}\#X_{\bullet})}$. So $\varphi$ is an algebra morphism.

It is left to show that $\varphi$ is a cochain map.
	\begin{align*}
		& (d_{\Hom} \circ \varphi (f \# g))(m \# x) \\
		= & d^{M\#X}(\varphi (f \# g)(m \# x)) - (-1)^{|f| + |g|} \varphi (f \# g)(d^{M\#X} (m \# x)) \\
		= & (-1)^{|g| |m|} d^{M\#X}\big(\sum_{(g)} f(g_{-1} \rightharpoonup m) \# g_0(x)\big)
		    - (-1)^{|f| + |g|} \varphi (f \# g)\big(d^{M}(m) \# x + (-1)^{|m|} m \# d^X(x)\big) \\
		= & (-1)^{|g| |m|} \big(\sum_{(g)} d^M(f(g_{-1} \rightharpoonup m)) \# g_0(x) + (-1)^{|f|+|m|} f(g_{-1} \rightharpoonup m) \# d^X(g_0(x))\big) \\
		& - (-1)^{|f| + |g|} \sum_{(g)} \big( (-1)^{|g|(|m| + 1)}f(g_{-1} \rightharpoonup d^{M}(m)) \# g_0(x) + (-1)^{|m| + |g||m|} f(g_{-1} \rightharpoonup m) \# g_0(d^X(x)) \big) \\
		= &\sum_{(g)} (-1)^{|g||m|} d_{\Hom}(f)(g_{-1} \rightharpoonup m) \# g_0(x) + (-1)^{|f| + (|g|+1)|m|} f(g_{-1} \rightharpoonup m) \# d_{\Hom}(g_0)(x) \\
		= & \varphi (d_{\Hom}(f) \# g + (-1)^{|f|} f \# d_{\Hom}(g))(m \# x) \\
		= & \varphi (d_{\Hom}(f \# g))(m \# x).
	\end{align*}
	Therefore, $d_{\Hom} \circ \varphi  = \varphi \circ d_{\Hom}$, and the conclusion follows.
\end{proof}

By taking the homology, then $\Ext^{\bullet}_{A^{op}}(M, M) = \bigoplus_{n \geq 0} \Ext^n_{A^{op}}(M, M)$ is a graded left $H$-module algebra, and $\Ext^{\bullet}_{B^{op}}(X, X)$ is a graded left $H$-comodule algebra. Hence we can form the graded Takeuchi smash product algebra $\Ext^{\bullet}_{A^{op}}(M, M) \# \Ext^{\bullet}_{B^{op}}(X, X)$.

\begin{thm}\label{Ext-alg of smash product}
	Suppose that $M \in {\mathcal{M}_{A\#H}}$ is pseudo-coherent as $A$-module, and $X \in {^H\mathcal{M}_B}$ is pseudo-coherent as $B$-module. Then, as graded algebras,
	$$\Ext^{\bullet}_{(A\#B)^{op}}(M\#X, M\#X) \cong \Ext^{\bullet}_{A^{op}}(M, M) \# \Ext^{\bullet}_{B^{op}}(X, X).$$
\end{thm}
\begin{proof}
	Let $P_{\bullet}$ be a projective resolution of $M_{A\#H}$. By Lemma \ref{relative-resolution}, there exists an exact sequence
	$$\cdots \longrightarrow Q_n \longrightarrow \cdots \longrightarrow Q_0 \longrightarrow X \longrightarrow 0 $$
	in ${^H\mathcal{M}_B}$, where all the $Q_n$'s are finitely generated free as $B$-module.
	
	By Proposition \ref{Hom-mod-comod-dga}, $\End^{\bullet}_{A^{op}}(P_{\bullet})$ is a differential graded left $H$-module algebra, and $\End^{\bullet}_{B^{op}}(Q_{\bullet})$ is a differential graded left $H$-comodule algebra, and
	$$\varphi: \End^{\bullet}_{A^{op}}(P_{\bullet}) \# \End^{\bullet}_{B^{op}}(Q_{\bullet}) \to \End^{\bullet}_{(A\#B)^{op}}(\mathrm{Tot}(P_{\bullet}\#Q_{\bullet}))$$
	$$f \# g \mapsto \big(p\#q \mapsto (-1)^{|g||p|}\sum_{(g)} f(g_{-1} \rightharpoonup p) \# g_0(q)\big)$$
	is a differential graded algebra morphism.
	By Proposition \ref{decom-Ext-group'} (3), it is a quasi-isomorphism.
	Then the conclusion follows by taking the homology.
\end{proof}

\section{Application to Takeuchi smash products of AS-regular algebras}
We recall some definitions first.
\subsection{Artin-Schelter regular algebras}
A graded algebra $A = \bigoplus_{n \in \Z} A_n$ is called {\it connected}, if $A_{<0} = 0$ and $A_0$ is the base field $k$. For any connected graded algebra $A$,  the left global dimension and right global dimension of $A$ are equal, which are equal to the projective dimension of the trivial module $k \cong A/A_{\geq 1}$ or $k_A$.

\begin{defn}
	A connected graded algebra $A$ is called {\it Artin-Schelter regular} (for short, AS-regular) of dimension $d$ for some integer $d \geq 0$, if
	\begin{enumerate}
		\item $A$ has global dimension $d$;
		\item $\Ext^i_{A}(k,A) \cong
		\begin{cases}
			0 & i \neq d \\
			k & i = d.
		\end{cases}$
	\end{enumerate}

%
	If $A$ is a left AS-regular algebra of dimension $d$, then it follows from the Ischebeck's spectral sequence that the right version of the condition (2) holds.
\end{defn}

Let $A$ be a connected graded algebra. By \cite[Lemma 1.2]{RRZ1}, $A$ is AS-regular of dimension $d$ if and only if it is a graded skew Calabi-Yau algebra of dimension $d$, which is defined in the following.

\begin{defn} A $k$-algebra $A$ is called {\it skew Calabi-Yau algebra} of dimension $d$, if
	\begin{itemize}
		\item [(i)] $A$ is {\it homologically smooth}, that is, as $A^e$-module, $A$ has a finitely generated projective resolution of finite length;
		\item [(ii)] there exists an automorphism $\mu$ of $A$, such that
		$$\Ext^i_{A^e}(A,A^e) \cong
		\begin{cases}{\tiny {\large }}
			0, & i \neq d\\
			A^{\mu}, & i = d
		\end{cases}$$
		as $A$-$A$-bimodules.
	\end{itemize}
The automorphism $\mu$ is unique up to an inner automorphism, which is called a {\it Nakayama automorphism} of $A$, denoted sometimes by $\mu_A$.

Graded skew Calabi-Yau algebras are defined similarly in the graded module category for graded algebras. 
\end{defn}
\subsection{Takeuchi smash products of AS-regular algebras}
A Hopf algebra $H = \bigoplus_{n \in \Z} H_n$ is said to be a {\it graded Hopf algebra} if both of its algebra and coalgebra structure are graded with respect to the grading $H = \bigoplus_{n \in \Z} H_n$, and its antipode is also graded.
Let $H$ be a graded Hopf algebra, and $A$ be a graded algebra which is also an $H$-module algebra. If $H_m \rightharpoonup A_n \subseteq A_{m+n}$ for all $n$ and $m$, then $A$ is called an {\it $H$-module graded algebra}.
Let $B$ be a graded algebra with an $H$-comodule algebra structure $\rho$. If $\rho(B_n) \subseteq \bigoplus_{i} H_i \otimes B_{n-i}$ for all $n$, then $B$ is called {\it an $H$-comodule graded algebra}. In this case,
it is clear that the Takeuchi smash product $A \# B$ is a graded algebra with $(A\#B)_n = \bigoplus_i A_i \otimes B_{n-i}$.

The Hopf algebra $H = k\langle g^{\pm 1}, X\rangle$ given in Example \ref{ore-ext-Hopf-alg} is a graded Hopf algebra by setting $\deg(g) = 0$ and $\deg(X) = 1$. Suppose $A$ is a graded algebra, with a graded automorphism $\sigma$ of degree $0$ and a graded $\sigma$-derivation $\delta$ of degree $1$. Then $A$ is an $H$-module graded algebra, $B=k[X]$ is an $H$-comodule graded algebra, and the Ore extension $A[x, \sigma, \delta]$ is the graded Takeuchi smash product $A \# B$.

\begin{thm}\label{AS-regular-TSP} Let $A$ and $B$ be AS-regular algebras of dimension $d_1$ and $d_2$ respectively. If $H$ is a graded Hopf algebra, $A$ is an $H$-module graded algebra and $B$ is an $H$-comodule graded algebra, such that $H \rightharpoonup A_{\geq 1} \subseteq A_{\geq 1}$ or $\rho(B_{\geq 1}) \subseteq H \otimes B_{\geq 1}$, then $A\#B$ is an AS-regular algebra of dimension $d_1+d_2$.
\end{thm}
\begin{proof}
	Suppose $H \rightharpoonup A_{\geq 1} \subseteq A_{\geq 1}$ (which is automatic if $H$ is positively graded). Then the trivial module $k \cong A/A_{\geq 1}$ is an $A\#H$-module.
	By Corollary \ref{proj-mod of smash product},  $\pd_{A\#B} \, k \leq \pd_A \, k + \pd_B \, k = d_1 + d_2$.
	It follows from the spectral sequence in Proposition \ref{Ext-spec-seq-of-smash-prod-mod} that
	$$\Ext^i_{A\#B}(k, A\#B) \cong \begin{cases}
		0, & i \neq d_1+d_2 \\
		\Ext^{d_2}_B(k, \Ext^{d_1}_A(k, A\#B)), & i = d_1+d_2.
	\end{cases}$$

	We claim that $\Ext^{d_1}_A(k, A\#B) \cong B$ as left $B$-modules. Let $P_{\bullet}$ be a projective resolution of $_{A\#H}k$. Then there is a natural $B$-module complex morphism
	$$\Hom_{A}(P_{\bullet}, A) \otimes B \To \Hom_A(P_{\bullet}, A\#B)$$
	where the $B$-module structure on $\Hom_{A}(P_i, A) \otimes B$ is given by $b \rightharpoonup (f \otimes b') = \sum_{(b)} (b_{-1} \rightharpoonup f) \otimes b_0 b'$, and  the $B$-module structure on $\Hom_{A}(P_i, A \# B)$ is given by \eqref{B-mod-on-Hom}. Since $_Ak$ is pseudo-coherent by \cite[Proposition 3.1]{SZ}, it follows that  $\Hom_{A}(P_{\bullet}, A) \otimes B \To \Hom_A(P_{\bullet}, A\#B)$ is a quasi-isomorphism of $B$-module complex. Thus, with the induced $B$-module structures, $\Ext^{d_1}_A(k, A\#B) \cong \Ext^{d_1}_A(k, A) \otimes B$ as $B$-modules.
	Since $\Ext_A^{d_1}(k, A)$ is an one-dimensional $k$-vector space, then there is algebra morphism $\alpha: H \to k$, such that $h \rightharpoonup e = \alpha(h) e$ for all $h \in H$ and $e \in \Ext_A^{d_1}(k, A)$.
	This implies that $\Ext_A^{d_1}(k, A\#B) \cong {^{\Xi^l_{\alpha}} B}$, where $\Xi^l_{\alpha} : B \to B$ is the algebra automorphism defined by $\Xi^l_{\alpha}(b) = \sum_{(b)} \alpha(b_{-1}) b_0$.
	Thus $\Ext_A^{d_1}(k, A\#B) \cong B$ as left $B$-modules, the claim is proved.
	
	Since $B$ is AS-regular, then $\Ext^{d_1+d_2}_{A\#B}(k, A\#B) \cong \Ext^{d_2}_B(k, B) \cong k$. It follows that $A\#B$ is a $(d_1+d_2)$-dimensional AS-regular algebra.
	
	Suppose $\rho(B_{\geq 1}) \subseteq H \otimes B_{\geq 1}$. Then $k \cong B/{B_{\geqslant 1}} \in ^H\!\mathcal{M}_B$. Then 
	it follows Corollary \ref{proj-mod of smash product'} and Proposition \ref{right-case-spec-seq}  that  $\pd_{(A\#B)^{op}} \, k \leq \pd_{A^{op}} \, k + \pd_{B^{op}} \, k = d_1 + d_2$ and
	$$\Ext^i_{(A\#B)^{op}}(k, A\#B) \cong \begin{cases}
		0, & i \neq d_1+d_2 \\
		\Ext^{d_1}_{A^{op}}(k, \Ext^{d_2}_{B^{op}}(k, A\#B)), & i = d_1+d_2.
	\end{cases}$$

    Since $A$ is AS-regular, by \cite[Proposition 3.1]{SZ}, $k_B$ is pseudo-coherent. It follows from Lemma \ref{relative-resolution} that $k \cong B/{B_{\geqslant 1}}$ has a  resolution $Q_\bullet \to k \to 0$ in $^H\!\mathcal{M}_B$ such that $Q_n = V_n \otimes B$ for some finite-dimensional left $H$-comodule $V_n$. Then, $\Hom_{B^{op}}(Q_n, A\#B)$ has an  $A^{op}$-module structure given by Lemma \ref{right-Hom-Tensor-adj}, that is, 
    $(fa)(y)=f(y_0)(S^{-1}y_{-1} \rightharpoonup a)$, and  $\Hom_{B^{op}}(Q_n, B)$ has a left $H$-comodule structure given by   $\rho(g)(y)=\sum_{(y),(g(y_0))} (g(y_0)_{-1} S^{-1}y_{-1}) \otimes g(y_0)_0$ as defined in \eqref{H-comod on right Hom set} (see Proposition \ref{decom-Ext-group'}).
    
    If we endow $A \otimes \Hom_{B^{op}}(Q_n, B)$ with a right $A$-module structure by
    $$(a \otimes g) a'= a( g_{-1}\rightharpoonup a') \otimes g_0.$$ 
    Then,
   $A \otimes \Hom_{B^{op}}(Q_n, B) \to \Hom_{B^{op}}(Q_n, A\#B), a \otimes g \mapsto \big([a, g]: Q \to A\#B, y \mapsto a\# g(y)\big)$
	is an $A^{op}$-module isomorphism. In fact.
    $$[a( g_{-1}\rightharpoonup a'), g_0](y)= a( g_{-1}\rightharpoonup a')  \# g_0(y)= a( g(y_0)_{-1} S^{-1}y_{-1}\rightharpoonup a') \# g(y_0)_0,\, \textrm{ and }$$
	$$([a, g] a')(y) =\big([a, g](y_0)\big) (S^{-1} y_{-1} \rightharpoonup a')=a\big(g(y_0)_{-1}S^{-1} y_{-1} \rightharpoonup a'\big) \#  g(y_0)_0.$$
It follows that $\Ext^{d_2}_{B^{op}}(k, A\#B) \cong A \otimes \Ext^{d_2}_{B^{op}}(k, B) \cong (A \otimes k)_A \cong A_A.$ Therefore, 
	$$\Ext^{d_1+d_2}_{(A\#B)^{op}}(k, A\#B) \cong \Ext^{d_1}_{A^{op}}(k, \Ext^{d_2}_{B^{op}}(k, A\#B)) \cong \Ext^{d_1}_{A^{op}}(k, A) \cong k.$$ 
Hence  $A\#B$ is an AS-regular algebra of dimension $d_1+d_2$.
\end{proof}

So graded Ore extensions preserve the AS-regularity, which  was proved in \cite[Proposition 2]{AST}. The ungraded case was proved in \cite{LWW} in terms of skew Calabi-Yau algebras.
The Yoneda Ext-algebras recover many properties of AS-regular algebras as it is well known (see, say \cite{LPWZ1, LPWZ2}), for instance,  the Nakayama automorphisms.

\begin{thm}\cite[Theorem 4.2]{RRZ2}
	Let $A$ be a noetherian AS-regular algebra generated in degree $1$. Let $\mu_A$ and $\mu_E$ be the Nakayama automorphism of $A$ and $\Ext^\bullet_A(k, k)$ respectively.
	Then $\mu_A|_{A_1} = (\mu_E|_{\Ext^{1}_A(k, k)})^*$ where $A_1$ is identified with $\Ext^{1}_A(k, k)^*$.
\end{thm}

As an application of Theorems \ref{Ext-alg of smash product} and \ref{AS-regular-TSP}, we can describe the Nakayama automorphism of the Takeuchi smash product of AS-regular algebras with a similar proof to Theorem \cite[Theorem 4.2]{SZL}.
For the definitions of homological determinants and codeterminant we refer \cite{KKZ} as the reference.

\begin{prop}
Suppose $A$ and $B$ are noetherian AS-regular algebras generated in degree $1$, and  $A$ is a graded $H$-module algebra, $B$ is a graded $H$-comodule for some Hopf algebra $H$. If $A\#B$ is noetherian, then the Nakayama automorphism $\mu_{A\#B}$ of $A\#B$ is given by
	$$\mu_{A\#B}(a\#b) = \sum_{(b)} \mu_A(g \rightharpoonup a) \# \hdet(b_{-1})\mu_B(b_0),$$
	where $\mu_A$ and $\mu_B$ are the Nakayama automorphisms of $A$ and $B$ respectively, $\hdet \in H^*$ and $g \in H$ are the homological determinant of the $H$-action on $A$ and the homological codeterminant of the $H$-coaction on $B$ respectively.
\end{prop}

\section*{Acknowledgements} This research is partially supported by the National Science Foundation of China (Grant No. 11771085) and the National Key Research and Development Program of China (Grant No. 2020YFA0713200).

%

\thebibliography{plain}

\bibitem[AST]{AST} M. Artin, W. Schelter, J. Tate, Quantum deformations of $GL_n$, Comm. Pure Appl. Math. 44 (1991), 879--895.

\bibitem[Aus]{Aus}
M. Auslander, On the dimension of modules and algebras. III. Global dimension, Nagoya Math. J. 9 (1955), 67--77.

\bibitem[BO]{BO} P. A. Bergh, S. Oppermann, Cohomology of twisted tensor products, J. Algebra 320 (2008), 3327--3338.

%
%

\bibitem[CG]{CG}
S. Caenepeel, T. Gu\'{e}d\'{e}non, Projectivity of a relative Hopf module over the subring of coinvariants, in: J. Bergen, S. Catoiu, W. Chin (Eds.), Hopf Algebras, in: Lect. Notes Pure Appl. Math., V. 237, Marcel Dekker, New York, 2004, 97--108.



\bibitem[CSV]{CSV} A. Cap, H. Schichl, J. Van\v{z}ura, On twisted tensor products of algebras, Comm. Algebra 23 (1995), 4701--4735.
%
%
%
%
%
%
%
%

\bibitem[KKZ]{KKZ}
E. Kirkman, J. Kuzmanovich, and J. J. Zhang, Gorenstein subrings of invariants under Hopf algebra actions, J. Algebra 322 (2009), 3640--3669.



\bibitem[LPWZ1]{LPWZ1} D.-M. Lu, J.H. Palmieri, Q.-S. Wu, J.J. Zhang, Regular algebras of dimension 4 and their A1 Ext-algebras, Duke Math. J. 137 (3) (2007) 537--584.

\bibitem[LPWZ2]{LPWZ2} D.-M. Lu, J.H. Palmieri, Q.-S. Wu, J.J. Zhang, A-infinity structure on Ext-algebras, J. Pure Appl. Algebra 213 (11) (2009) 2017--2037.

\bibitem[LWW]{LWW}
L.-Y. Liu, S.-Q. Wang, and Q.-S. Wu. Twisted Calabi-Yau property of Ore extensions, J. Noncomm. Geo. 8 (2012), 587--609.

\bibitem[LWZ]{LWZ}	
L.-Y. Liu, Q.-S. Wu, and C. Zhu, Hopf action on Calabi-Yau algebras, New Trends in Noncommutative Algebra,
{\it Contemp. Math.} 562 (2012), 189--210.


%

\bibitem[Mar]{Mar} R. Mart\'{i}nez-Villa, Skew group algebras and their Yoneda algebras, Math. J. Okayama Univ. 43 (2001), 1--16.



\bibitem[Mo]{Mo}
S. Montgomery, Hopf algebras and their actions on rings.  CBMS Regional Conference Series in Mathematics 52 (1993).


%
%

\bibitem[RRZ1]{RRZ1}
M. Reyes, D. Rogalski, and J. J. Zhang, Skew Calabi-Yau algebras and homological identities, Adv. Math. 264 (2014), 308--354.
	
\bibitem[RRZ2]{RRZ2}
M. Reyes, D. Rogalski, and J. J. Zhang, Skew Calabi-Yau triangulated categories and Frobenius ext-algebras, Trans. Amer. Math. Soc. 369 ( 2017), 309--340.

\bibitem[Sw]{Sw}
M. Sweedler, Hopf Algebras, Benjamin, New York, 1969.


\bibitem[SZL]{SZL}
Y. Shen, G.-S. Zhou, and D.-M. Lu, Nakayama automorphisms of twisted tensor products, J. Algebra 504 (2018), 445--478.

%

\bibitem[SZ]{SZ} D. R. Stephenson, J. J. Zhang, Growth of graded Noetherian rings, Proc. Amer. Math. Soc. 125 (1997), 1593--1605.

\bibitem[Tak]{Tak}
M. Takeuchi, $\mathrm{Ext}_{\mathrm{ad}}(\mathrm{Sp} R,\mu^A) \cong \hat{\mathrm{Br}}(A/k)$, J. Algebra 67 (1980), 436--475.

\bibitem[Ul]{Ul}
K.-H. Ulbrich, Smash products and comodules of linear maps, Tsukuba J. Math. 14 (1990), 371--378.

%


\bibitem[WZ]{WZ}
Q.-S. Wu, C. Zhu, Skew group algebras of Calabi-Yau algebras, J. Algebra 340 (2011), 53--76.
\end{document}